\documentclass[11pt]{amsart}
\usepackage{amsmath,amscd,amssymb,enumerate,verbatim,palatino,parskip}
\usepackage[T1]{fontenc}
\usepackage{hyperref}
\usepackage{graphicx}
\usepackage{pinlabel}
\usepackage{tikz}
\usetikzlibrary{matrix}

\newcommand{\CC}{\mathbf{C}}

\newcommand{\RR}{\mathbf{R}}

\newcommand{\ZZ}{\mathbf{Z}}

\newcommand{\id}{\mathrm{Id}}

\numberwithin{equation}{section}

\newtheorem{thm}[equation]{Theorem}
\newtheorem{dfn}[equation]{Definition}
\newtheorem{lma}[equation]{Lemma}
\newtheorem{prp}[equation]{Proposition}
\newtheorem{cor}[equation]{Corollary}

\newtheorem{rmk}[equation]{Remark}

\newtheoremstyle{TheoremNum}
    {\topsep}{\topsep}              %%% space between body and thm
    {\itshape}                      %%% Thm body font
    {}                              %%% Indent amount (empty = no indent)
    {\bfseries}                     %%% Thm head font
    {.}                             %%% Punctuation after thm head
    { }                             %%% Space after thm head
    {\thmname{#1}\thmnote{ \bfseries #3}}%%% Thm head spec
\theoremstyle{TheoremNum}

\begingroup 
\makeatletter 
\@for\theoremstyle:=definition,remark,plain,TheoremNum\do{% 
\expandafter\g@addto@macro\csname th@\theoremstyle\endcsname{% 
\addtolength\thm@preskip\parskip 
}% 
} 
\endgroup 

\title{Exact Lagrangian caps and non-uniruled Lagrangian submanifolds}

\author{Georgios Dimitroglou Rizell}
\address{Universit\'{e} Paris-Sud, D\'{e}partement de Math\'{e}matiques, Bat. 425, 91405 Orsay, France}
\email{georgios.dimitroglou@math.u-psud.fr}
\urladdr{https://sites.google.com/site/georgiosdimitroglourizell/}
\thanks{This work was partially supported by the ERC Starting Grant of Fr{\'e}d{\'e}ric Bourgeois StG-239781-ContactMath.}

\begin{document}
\begin{abstract}
We make the elementary observation that the Lagrangian submanifolds of $\CC^n$, for each $n \ge 3$, constructed by Ekholm, Eliashberg, Murphy and Smith in \cite{ImmersionsFew} are non-uniruled and moreover have infinite relative Gromov width. The construction of these submanifolds use exact Lagrangian caps, which obviously are non-uniruled in themselves. This property is also used to show that if a Legendrian submanifold inside a 1-jet space admits an exact Lagrangian cap then its Legendrian contact homology DGA is acyclic.
\end{abstract}
\maketitle
\section{Introduction}
\subsection{Background}
A \emph{contact manifold} $(Y,\lambda)$ is a smooth $2n+1$-dimensional manifold with a maximally non-integrable hyper-plane field which here is given by $\ker \lambda$, where $\lambda$ is the so-called contact one-form. It follows that $\lambda \wedge (d\lambda)^{\wedge n}$ is a volume form on $Y$. A \emph{Legendrian submanifold} $\Lambda \subset Y$ is an $n$-dimensional submanifold which is tangent to $\ker \lambda$.

A \emph{symplectic manifold} $(X,\omega)$ is a $2n$-dimensional manifold $X$ together with a closed non-degenerate two-form $\omega$. A \emph{Lagrangian submanifold} $L \subset X$ is an $n$-dimensional submanifold satisfying $\omega|_{TL}=0$.

By a \emph{Lagrangian cap} of $\Lambda$ we mean a properly embedded Lagrangian submanifold (without boundary)
\[L_{\Lambda,\emptyset} \subset (\RR \times Y,d(e^t\lambda))\]
of the symplectisation of $(Y,\lambda)$ coinciding with half a cylinder $(-\infty,N] \times \Lambda$ outside of some compact  set. Here $t$ denotes the coordinate on the $\RR$-factor. We say that the cap is exact if the pull-back of the one-form $e^t\lambda$ to the cap is exact and, moreover, has a primitive which vanishes on the negative end.

Let $(X,d\alpha)$ be a compact Liouville domain with contact boundary $(Y:=\partial X,\lambda)$ and consider its completion $(\overline{X},d\alpha)$ formed by adjoining the non-compact end $(([0,+\infty) \times Y,d(e^t \lambda))$ consisting of half a symplectisation to $(X,d\alpha)$. We are interested in \emph{Lagrangian fillings} $L_{\emptyset,\Lambda} \subset \overline{X}$ of $\Lambda$, i.e.~a proper embedded Lagrangian submanifold (without boundary) coinciding with half a cylinder $[N,+\infty) \times \Lambda$ outside of a compact set.

Given a filling $L_{\emptyset,\Lambda}$ and a cap $L_{\Lambda,\emptyset}$ of $\Lambda$ as above one can construct their \emph{concatenation} as follows. After a translation of $L_{\Lambda,\emptyset}$ in the $t$-direction, which is an isotopy through exact Lagrangian submanifolds, we may suppose that there is some $N>0$ for which
\[L_{\emptyset,\Lambda} \cap \{ t \ge N-1 \} = [N-1,+\infty) \times \Lambda \]
in the cylindrical end of $\overline{X}$, while
\[ L_{\Lambda,\emptyset} \cap \{ t \le N\}=(-\infty,N] \times \Lambda \]
inside the symplectisation. We define the concatenation to be the closed Lagrangian submanifold
\[ L = ( L_{\emptyset,\Lambda} \cap \{ t \le N \}) \cup  (L_{\Lambda,\emptyset} \cap \{ t \ge N\}) \subset \overline{X}\]
where $ L_{\Lambda,\emptyset} \cap \{ t \ge N\}$ has been canonically identified with an exact Lagrangian submanifold (with boundary) of the cylindrical end of $\overline{X}$. The concatenation is thus the simultaneous resolution of the conical singularities of both $\RR \times Y$ and $L_{\Lambda,\emptyset} \subset \RR \times Y$. Observe that we do not assume the filling to be exact, and hence neither is the concatenation in general. 

We say that a Lagrangian immersion is \emph{displaceable} if there exists a time-dependent Hamiltonian isotopy for which the image of the immersion is disjoint from its image composed with the time-one map of the Hamiltonian flow.

We will also need the following notions.

Consider a Lagrangian submanifold $L \subset (X,\omega)$ of a general symplectic manifold and let $\mathcal{B}(X,L,r)$ be the set of symplectic embeddings of $B^{2n}(r) \subset (\CC^n,\omega_0)$ into $X$ mapping the real-part into $L$ and which otherwise has image disjoint from $L$. The following notion was first was considered in \cite{HomotopyDynamics} by Barraud and Cornea.
\begin{dfn}
The \emph{(relative) Gromov width} of a Lagrangian submanifold $L \subset X$ is given by
\[ w(L,X) := \sup\{ \pi r^2 \in [0,+\infty) ; \:\: \mathcal{B}(X,L,r) \ne \emptyset \}. \]
\end{dfn}
The Weinstein Lagrangian neighborhood theorem implies that the Gromov width is positive and it is obviously invariant under Hamiltonian isotopy.

\begin{dfn}
A Lagrangian immersion $L \subset (X,\omega)$ whose self-intersections consist of transverse double-points is said to be \emph{uniruled} there exists some $A>0$ for which there is a Baire subset of compatible almost complex structures $J$ such that, for any given $x \in L$, there is a non-constant $J$-holomorphic disc in $X$ of $\omega$-area at most $A$ and which has boundary on $L$ passing through $x$.
\end{dfn}

Uniruledness is known to imply finiteness of the Gromov width, see for instance \cite[Corollary 3.10]{LagIntSerre}. It has been expected that every displaceable Lagrangian submanifold is uniruled \cite[Conjecture 3.15]{HomotopyDynamics}. The observations in this paper shows that this conjecture is false, since it is not satisfied by the Lagrangian submanifolds of $\CC^n$ constructed in \cite{ImmersionsFew} by Ekholm, Eliashberg, Murphy and Smith. See Corollary \ref{cor:width} below.

For another source of counter examples, see \cite{ClosedExactLag} where Murphy constructs  exact closed Lagrangian submanifolds inside certain symplectisations (which hence also are displaceable but not uniruled).

\begin{dfn}
Let $L \subset (\overline{X},d\alpha)$ be a Lagrangian submanifold of a simply connected symplectic manifold with vanishing Chern class. Let $\sigma \in H^1(L;\RR)$ denote the pull-back of $\alpha$ to $L$, called the symplectic action, and let $\mu \in H^1(L,\ZZ)$ be the Maslov class of $L$. If
\[\sigma=K\mu\]
for some $K>0$, we say that $L$ is \emph{monotone}.
\end{dfn}

Previous results show that uniruledness indeed holds for many classes of displaceable Lagrangian submanifolds. For instance, it was shown for certain monotone Lagrangian submanifolds by Biran and Cornea \cite{RigidityUniruling}, a result which in particular applies to  monotone Lagrangian submanifolds of $\CC^n$. Subsequent work of Charette \cite{GeomRefinement} proves \cite[Conjecture 3.15]{HomotopyDynamics} for monotone Lagrangian submanifolds. Also, see Cornea and Lalonde \cite{ClusterHomology}, and Zehmisch \cite{CodiscRadius} for similar results.

Note that Theorem \ref{thm:uniruled} below also is a variant of such a result.

Finally, in forthcoming work by Borman and McLean \cite{WidthCapacities} the finiteness of the Gromov width is shown for a large class of, not necessarily monotone, Lagrangian submanifolds of $\CC^n$ characterised by the topological property of admitting a metric with non-positive scalar curvature. More precisely, the Gromov width is shown to be bounded below by four times the displacement energy for such a Lagrangian submanifold.

\subsection{Results}
\subsubsection{Counterexamples to \cite[Conjecture 3.15]{HomotopyDynamics}}
We first make the following elementary observation.
\begin{prp}
\label{prop:width}
Let $L \subset \overline{X}$ be a closed Lagrangian submanifold which is on the form
\[L \cap ([N,N+\epsilon] \times Y)=[N,N+\epsilon] \times \Lambda,\]
for some $N \ge 0$. If $e^t \lambda$ is exact on $L_{cap}:=L \cap \{ t \ge N \}$ and if a primitive can be chosen which vanishes along its boundary, it follows that
\begin{itemize}
\item $L$ is not uniruled,
\item $L$ has infinite relative Gromov width.
\end{itemize}
\end{prp}
In particular the assumptions above are satisfied for $L$ being the concatenation of an exact Lagrangian cap and a Lagrangian filling.

Observe that the construction of displaceable Lagrangian submanifolds satisfying the assumptions of Proposition \ref{prop:width} is highly non-trivial. The only examples in $\CC^n$ known to the author are the examples constructed in \cite{ImmersionsFew}. Their construction provides a variety of examples in $\CC^n$ for each $n \ge 3$. However, note that these examples all are non-orientable when $n$ is even (see Remark \ref{rem:orient}). We outline the construction in Section \ref{sec:constr}.

In particular, they provide non-monotone Lagrangian embeddings of $S^1 \times S^{2k}$ into $\CC^{1+2k}$ for $k>0$ satisfying these assumptions. Using these constructions, it is easy to show
\begin{cor}
\label{cor:width}
Let
\[S:=S^{i_1} \times \hdots \times S^{i_m}\]
be an $M$-dimensional product of spheres, where $i_j \ge 0$. For each $k>0$ there exists a Lagrangian embedding of $S^1 \times S^{2k} \times S$ into $\CC^{1+2k+M}$ having infinite Gromov width.
\end{cor}
\begin{proof}
The symplectic manifold
\begin{eqnarray*}
\lefteqn{(\CC^n \times T^*(S^l),\omega_0 \oplus d\theta_{S^l})} \\
& \simeq& (T^*(\RR^n) \times T^*(S^l),d\theta_{\RR^n}\oplus d\theta_{S^l}) \\
& \simeq & (T^*(\RR^n \times S^l),d\theta_{\RR^n \times S^l} ) \\
& \subset & (T^*(\RR^{n+l}),d\theta_{\RR^{n+l}} )
\end{eqnarray*}
embeds into $\CC^{n+l}$, where $\theta_N$ denotes the Liouville one-form on $T^*N$. Taking a product of a Lagrangian submanifold $L \subset \CC^n$ with the zero section in $T^*(S^l)$ thus produces a Lagrangian embedding of $L \times S^l$ in $\CC^{n+l}$.

Observe that if $L_1 \subset (X_1,\omega_1)$ and $L_2 \subset (X_2,\omega_2)$ have infinite Gromov width, the same is true for $L_1 \times L_2 \subset (X_1 \times X_2, \omega_1 \oplus \omega_2)$. The result now follows from the fact that the zero-section inside a cotangent bundle has infinite Gromov width together with the observations that the construction in \cite{ImmersionsFew} of a Lagrangian embedding of $S^1 \times S^{2k}$ inside $\CC^{1+2k}$ satisfy Proposition \ref{prop:width}.
\end{proof}
This construction is related to the front-spinning construction found Ekholm, Etnyre and Sullivan \cite[Section 4.4]{NonIsoLeg} and generalised by Golovko \cite{NoteSpin}. See Remark \ref{rem:spin} below.

It is interesting to note that the situation is very different for Lagrangian embeddings of $S^1 \times S^{2k+1}$ in $\CC^{2+2k}$ for $k>0$. Namely, work by Fukaya and Oh \cite[Proposition 2.10]{ApplicationsFloer} implies that such a Lagrangian submanifold is monotone. Also, see \cite[Corollary 4.6]{ClusterHomology} for a proof of uniruledness of these Lagrangian submanifolds.

Finally, a sufficiently stabilised Legendrian knot inside $(S^{3},\lambda_{std})$ admits a Lagrangian cap in the symplectisation by a result due to Lin \cite{ExactCapsKnots}. These caps can be used to construct \emph{non-orientable} Lagrangian submanifolds of $\CC^2$ satisfying the assumptions of Proposition \ref{prop:width}. This construction is related to the exact Lagrangian immersions constructed explicitly by Sauvaget in \cite{Sauvaget}.

\subsubsection{Implications for the Legendrian contact homology of $\Lambda$} We will in the following consider contact manifolds
\[(Y=P \times \RR,\lambda=dz+\theta)\]
being contactisations of an exact symplectic manifold $(P,d\theta)$, where $z$ denotes a coordinate on the $\RR$-factor. We will moreover assume that $(P,d\theta)$ is symplectomorphic to the completion of a Liouville domain. Observe that a Legendrian submanifold of $P \times \RR$ projects to an exact Lagrangian immersions in $P$ and, conversely, a generic exact Lagrangian immersion lifts to a Legendrian submanifold of $P \times \RR$.

Important examples of contactisations of the above form is the contact manifold
\[(J^1(M)=T^*M \times \RR,dz+\theta_M),\]
where $\theta_M$ is the Liouville form on $T^*M$, as well as the standard contact $(2n+1)$-space
\[(J^1(\RR^n)=(\RR^n \times \RR^n \times \RR),dz+\theta_{\RR^n}=dz+p_i dq^i).\]

Displaceability together with non-obstructedness for different Floer theories associated to a Lagrangian submanifold has in many cases been shown to imply uniruledness. For instance, see \cite[Corollary 1.18]{ClusterHomology}. Here we will consider the Legendrian contact homology differential graded algebra (DGA for short), also called the Chekanov algebra, associated to the Legendrian lift of an exact Lagrangian immersion (see Section \ref{sec:lch} below for more details).

An augmentation is a unital DGA-morphism to the coefficient ring (considered as a trivial DGA), and the existence of an augmentation should be viewed as a certain non-obstructedness of the Legendrian contact homology. The following proposition (and its proof) is a direct generalisation of \cite[Theorem 5.5]{DualityLeg}, a result which can be translated into the statement that if $L$ is displaceable and if the Chekanov DGA of its Legendrian lift has an augmentation, then $L$ is uniruled. 
\begin{thm}
\label{thm:uniruled}
Let $R$ be a unital ring and $L \subset (P,d\theta)$ a displaceable exact Lagrangian immersion. If $1+1 \neq 0$ in $R$ then we moreover assume that $L$ is spin, and we fix the choice of a spin structure. Suppose that the Chekanov DGA with coefficients in $R$ of the Legendrian lift of $L$ is not acylic (with or without Novikov coefficients), then $L$ is uniruled.
\end{thm}
\begin{rmk} Because of the strict non-commutativity of the DGA under consideration, admitting an augmentation is strictly stronger than being acyclic. See \cite{IsotopiesKnots} for examples.
\end{rmk}

In Section \ref{sec:double} below, assuming that $\Lambda \subset (J^1(M),dz+\theta_M)$ admits an exact Lagrangian cap in the symplectisation, we construct a displaceable exact Lagrangian immersion
\[ L_\Lambda \subset (\RR \times J^1(M),d(e^t(dz+\theta_M)))\]
of a closed manifold whose Legendrian contact homology DGA without Novikov coefficients is homotopy equivalent to that of $\Lambda$ by Lemma \ref{lem:equiv}. In other words
\begin{cor}
\label{cor:acyclic}
Suppose that
\[L_{\Lambda,\emptyset} \subset (\RR \times J^1(M),d(e^t(dz+\theta_M)))\]
is an exact Lagrangian cap with a negative end being a half cylinder over the Legendrian submanifold
\[\Lambda \subset (J^1(M),dz+\theta_M).\]
It follows that the Legendrian contact homology DGA of $\Lambda$ with $\ZZ_2$-coefficients (without Novikov coefficients) is acyclic.

If $L_{\Lambda,\emptyset}$ is a spin cobordism, it follows that the same is true with $\ZZ$-coefficients, given that the DGA of $\Lambda$ is induced by a choice of spin structure on the cap.
\end{cor}
\begin{rmk}
The exact caps constructed in \cite{LagCaps} all have a negative end being a half cylinder over a so-called loose Legendrian submanifold, see \cite{LooseLeg}. By an explicit computation (\cite[Section 8]{LooseLeg} and \cite[Section 4.3]{NonIsoLeg}) it follows that the DGA of a loose Legendrian submanifold, either using Novikov coefficients or not, is acyclic.
\end{rmk}
\begin{rmk} There is still room for a non-trivial Legendrian contact homology \emph{with Novikov coefficients} for a Legendrian submanifold satisfying the above assumptions. The knotted Legendrian torus $L_{1,1} \subset J^1(\RR^2)$ constructed in \cite{KnottedLeg} has an acyclic DGA without Novikov coefficients, but its DGA is not acyclic when using Novikov coefficients. However, it can be seen that there is a proper embedded exact Lagrangian cobordism $V \subset \RR \times J^1(\RR^2)$ inside the symplectisation coinciding with
\[((-\infty,-N] \times L_{1,1}) \cup ([N,+\infty) \times \Lambda_0),\]
for some $N>0$, outside of a compact set, where $\Lambda_0$ is the \emph{loose Legendrian two-sphere} (see \cite{LooseLeg}). Since the loose two-sphere admits an exact Lagrangian cap by \cite{LagCaps}, so does $L_{1,1}$ by concatenating the cap with the above cobordism. In particular, admitting a cap does not imply being loose.

To construct $V$ one can proceed as follows. A Legendrian knot $K \subset J^1(\RR)$ which is symmetric under a reflection
\[(q,p,z) \mapsto (C-q,-p,z)\]
can be rotated around its point of symmetry to form an embedded Legendrian surface
\begin{gather*}
S^1 \times K \to J^1(\RR^2) \simeq \RR^2 \times \RR^2 \times \RR\\
(\theta,k) \mapsto \left( ((q(k)-C/2)\cos \theta,(q(k)-C/2)\sin \theta), (p(k) \cos \theta,p(k)\sin \theta),z(k)\right)
\end{gather*}
whose image is either a torus or a sphere. $L_{1,1}$ is obtained as the rotation of the Legendrian knot in $J^1(\RR)$ shown in Figure \ref{fig:torus}. Rotating the Legendrian one-disk $D$ shown in the same figure we obtain a Legendrian two-disk with boundary on $L_{1,1}$. This Legendrian disk determines an ambient Legendrian one-sugery on $L_{1,1}$ as defined in \cite{LegAmbient}, which produces the rotation of the Legendrian knot shown in Figure \ref{fig:loose}. The resulting Legendrian surface is in fact the loose two-sphere $\Lambda_0$. Finally, the construction of an ambient Legendrian surgery in \cite{LegAmbient} provides an exact Lagrangian cobordism $V$ as required which moreover is diffeomorphic to the corresponding two-handle attachment.
\end{rmk}

\begin{figure}[htp]
\centering
\labellist
\pinlabel $D$ at 82 20
\pinlabel $S$ at 74 43

\pinlabel $q$ at 202 75
\pinlabel $z$ at 169 108
\pinlabel $q$ at 202 8
\pinlabel $p$ at 169 43
\endlabellist
\includegraphics{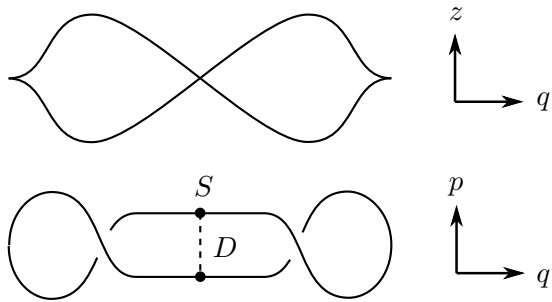}
\caption{Before the surgery.}
\label{fig:torus}
\end{figure}

\begin{figure}[htp]
\centering
\labellist
\pinlabel $q$ at 202 74
\pinlabel $z$ at 169 107
\pinlabel $q$ at 202 8
\pinlabel $p$ at 170 43
\endlabellist
\includegraphics{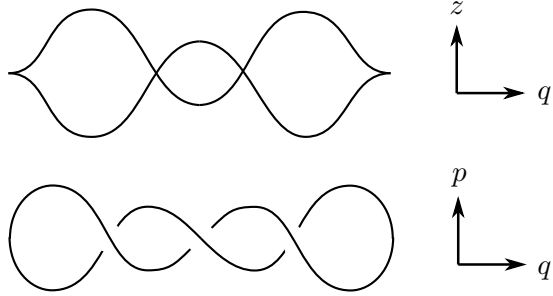}
\caption{After the surgery.}
\label{fig:loose}
\end{figure}

\begin{rmk}
The construction in Section \ref{sec:double} can be generalised to work for a Legendrian submanifold in a general contactisation $(P \times \RR, dz+\theta)$ admitting an exact Lagrangian cap in the symplectisation. The same computation should work in this setting as well, thus proving the above corollary for a general contactisation of a Liouville domain.
\end{rmk}

\subsection{Existence of exact Lagrangian caps and constructions}
\label{sec:constr}
In \cite{LooseLeg} Murphy introduced the concept of a \emph{loose} Legendrian $n$-dimensional submanifold for $n \ge 2$. Moreover, loose Legendrian submanifolds where shown to satisfy an h-principle.

Assume that we are given a loose Legendrian submanifold $\Lambda \subset (Y,\lambda)$. Using the above h-principle, it was shown in \cite{LagCaps} by Eliashberg and Murphy that, under some homotopy-theoretic assumptions, an exact immersed Lagrangian cap of $\Lambda$ is Lagrangian regular homotopic through exact immersions to an embedded exact Lagrangian cap. Recall that exact Lagrangian immersions satisfy an h-principle by Gromov \cite{PartialDiffRel} and Lees \cite{ClassificationLagrangeImmersions} but that there is no such h-principle for general Lagrangian embeddings. 

Ekholm, Eliashberg, Murphy and Smith \cite{ImmersionsFew} used the above h-principle for exact Lagrangian caps of loose Legendrian submanifolds to construct many interesting examples of Lagrangian embeddings inside $\CC^n$ for each $n \ge 3$. These examples are constructed by considering an embedded exact Lagrangian cap inside
\[(\RR \times S^{2n-1}, d(e^t \lambda_{std})) \]
where $\lambda_{std}:=\frac{1}{2}(x_idy^i-y_idx^i)$ is the standard contact one-form on $S^{2n-1} \subset \CC^n$. This cap is then concatenated 
with a (non-exact) Lagrangian filling inside 
\[\left(\CC^n,\omega_0=d\left(\frac{1}{2}(x_idy^i-y_idx^i)\right)\right).\]
Observe that the result is a closed Lagrangian submanifold of the standard symplectic $(\CC^n,\omega_0)$ which hence is displaceable. Their constructions were shown to satisfy many surprising properties.

We now describe their examples inside $\CC^3$. Starting with the loose Legendrian two-sphere $\Lambda_0 \subset (S^5,\lambda_0)$ and given any three-manifold $M$, it follows from the theory in \cite{LagCaps} that there is an exact Lagrangian cap inside the symplectisation $(\RR \times S^5,d(e^t\lambda_{std}))$ which is diffeomorphic to $M \setminus \{ pt\}$ and which moreover has vanishing Maslov class.

A standard construction (see \cite{ImmersionsFew}) shows that $\Lambda_0$ also admits an exact \emph{immersed} filling with a single double-point. Observe that this filling indeed must have double-points, since the concatenation of the cap and the filling otherwise would be a closed exact Lagrangian submanifold of $\CC^3$, thus contradicting Gromov's theorem \cite{Gromov}. Finally, it can be checked that the Legendrian contact homology grading of this double-point must be equal to one in this case. We refer to \cite{ContHomR} for the definition of this grading.

The constructed exact Lagrangian immersion of $M$ is itself interesting, since their number of double-points in general are not bounded from below by $\frac{1}{2}\dim H(M;\ZZ_2)$. By \cite[Theorem 1.2]{DualityLeg} this bound holds for the number of transverse double-points for any exact Lagrangian immersion in $\CC^n$ whose Legendrian lift has a Legendrian contact homology DGA admitting an augmentation.

The above double-point of the exact immersed Lagrangian filling of $\Lambda_0$ can be removed by a Polterovich surgery defined in \cite{LagrangeSurgery}. A suitable choice of such a surgery  produces a (non-exact) embedded Lagrangian filling $L_{\emptyset,\Lambda_0}$ of $\Lambda_0$ diffeomorphic to $S^2 \times S^1 \setminus \{pt\}$. The fact that the grading of the double-point is equal to one implies that $L_{\emptyset,\Lambda_0}$ has vanishing Maslov class for this choice of Polterovich surgery.

In conclusion the construction produces, for every three-manifold $M$, a Lagrangian embedding of $M \# (S^2 \times S^1)$ in $\CC^3$ with vanishing Maslov class. These examples provided the first known examples of Lagrangian submanifolds of $\CC^n$ with this property. Moreover, by construction, these Lagrangian submanifolds satisfy the assumptions of Proposition \ref{prop:width}.

In higher dimensions, the same technique give Lagrangian embeddings of $S^1 \times S^{2k}$ in $\CC^{1+2k}$ for $k>0$, see \cite[Corollary 1.6]{ImmersionsFew}. Furthermore, the Maslov class evaluated on the unique generator $\gamma \in H_1(S^1 \times S^{2k};\ZZ)$ having positive symplectic action takes the value $2k-2$ for these embeddings.

\begin{rmk}
\label{rem:orient} A Polterovich surgery on a non-uniruled exact Lagrangian immersion $L \subset \CC^{n}$ produces a non-orientable Lagrangian submanifold when $n=2k$. To see this, observe that the Polterovich surgery on a double-point of odd grading always creates a non-orientable submanifold in these dimensions. Moreover, the non-uniruledness of $L$ together with Theorem \ref{thm:uniruled} implies that there must be a double-point of odd degree, since otherwise the DGA would not be acylic. To that end, observe that the differential decreases the degree by one and that the unit is of degree zero.
\end{rmk}

\begin{rmk}
\label{rem:spin}
The examples produced by Corollary \ref{cor:width} are constructed by a version of the front-spinning construction, but it can be seen that they still are obtained by the concatenation of a cap and filling. Namely, suppose that we are given an exact Lagrangian cap 
\[L_{\Lambda,\emptyset} \subset \RR \times J^1(\RR^n)\]
and a non-exact Lagrangian filling
\[L_{\emptyset,\Lambda} \subset \RR \times J^1(\RR^n).\]
The front-spinning construction produces a Legendrian embedding of $\Lambda \times S^l$ inside $J^1(\RR^{n+l})$. This spinning extends to the symplectisation by \cite{NoteSpin}, producing an exact Lagrangian cap
\[L_{\Lambda \times S^l,\emptyset} \subset \RR \times J^1(\RR^{n+l})\]
and a non-exact Lagrangian filling
\[L_{\emptyset,\Lambda \times S^l} \subset \RR \times J^1(\RR^{n+l}),\]
diffeomorphic to $L_{\Lambda,\emptyset} \times S^l$ and $L_{\emptyset,\Lambda} \times S^l$, respectively.
\end{rmk}

\section{Proof of Proposition \ref{prop:width}}
Consider a smooth cut-off function $\rho \colon [0,+\infty) \to \RR$ which satisfies $\rho'(t) \ge 0$ for all $t$, $\rho(t)=0$ for $t \le N$, and $\rho(t)=1$ for $t \ge N+\epsilon$.
The flow
\begin{gather*}
\phi^s \colon [0,+\infty) \times Y \to [0,+\infty) \times Y\\
(t,y) \mapsto (t+\rho(t)s,y)
\end{gather*}
on the cylindrical end of $\overline{X}$ satisfies
\[ (\phi^s)^* (e^t\lambda) =e^{\rho(t)s}e^t\lambda \]
and coincides with the Liouville flow on the set $[N+\epsilon,+\infty) \times Y$.

The condition that
\[ [N,N+\epsilon] \times \Lambda \subset (\RR \times Y,d(e^t\lambda))\]
is a Lagrangian submanifold is equivalent to the pull-back of $e^t\lambda$ being closed. Since
\[ d(e^t\lambda)=e^t dt \wedge \lambda +e^t d\lambda\]
it follows that the pull-back of $\lambda$ to this Lagrangian cylinder actually vanishes. This shows that $(\phi^s)^*(e^t \lambda)$ is closed on $L_{cap}$ and, since $\phi^s=\id$ when restricted to $L \setminus L_{cap}$, it follows that
\[L_s := \phi^s(L)\]
is an isotopy through Lagrangian submanifolds.

Moreover, writing $f \colon L \to \RR$ for the primitive of $e^t\lambda$ pulled-back to $L_{cap}$ vanishing along the boundary, the function $e^s f$ is obviously a primitive of $e^t\lambda$ pulled back to $L_{cap}$ under $\phi^s$ (which thus also vanishes along the boundary). Since the $\phi^s$ has support in the complement of $L \setminus L_{cap}$, it follows that 
\[ (\phi^s|_L)^*(\alpha) = (\phi^0|_L)^*\alpha + d((e^s-1)f)\]
is a path of cohomologous one-forms. By the Weinstein Lagrangian neighbourhood theorem it now follows that the isotopy $L_s$ may be realised by a time-dependent Hamiltonian isotopy. 

Assume that $L$ is uniruled, i.e.~that there is a Baire set of compatible almost complex structures on $\overline{X}$ for which there are non-constant pseudo-holomorphic discs passing through every point of $L$ and having symplectic area bounded by $A>0$.

First, observe that the above exactness property rules out the existence of $J$-holomorphic discs with boundary on $L$ contained entirely in $\{ t \ge N\}$. Take a compatible almost complex structure $J$ which is invariant under translations of the $t$-coordinate on the set $\{t \ge N\}$. The monotonicity property for the symplectic area of $J$-holomorphic discs with boundary on a Lagrangian submanifold (see \cite{SomeProp}) shows that there is a sufficiently large $s \gg 0$ for which all $J$-holomorphic discs in $\overline{X}$ with boundary on $L_s$ and boundary points passing through both $L_s \cap \{ t=N \}$ and $L_s \cap \{ t =N+s \}$ have symplectic area at least $A+1$. Observe that the same is true for any other almost complex structure sufficiently $C^l$-close to $J$ for any $l>2$. Since $L_s$ is Hamiltonian isotopic to $L$, this leads to a contradiction.

For the statement concerning the Gromov width, observe that the Weinstein Lagrangian neighborhood theorem gives a symplectic embedding of the standard symplectic ball of some positive radius $r>0$ intersecting $L$ precisely in the real part and which is contained in the set $\{ t \ge N+\epsilon\}$. By composing this embedding with $\phi^s$, we see that $L_s$ admits such a symplectic ball of radius $e^sr$. Finally, since $L_s$ is Hamiltonian isotopic to $L$, it follows that $L$ admits a symplectic ball of radius $e^sr$ as well. In conclusion, $L$ has infinite Gromov width.

\section{The Legendrian contact homology DGA of the negative end of an exact Lagrangian cap.}
A path of Legendrian embeddings is called a \emph{Legendrian isotopy}. Legendrian contact homology is an algebraic Legendrian isotopy invariant introduced in \cite{IntroSFT} by Eliashberg, Hofer and Givental and in \cite{DiffAlg} by Chekanov. Here we will use the version defined in \cite{OrientLeg} and \cite{ContHomP} by Ekholm, Etnyre and Sullivan, which is defined for the standard contact form on the contactisation $(P \times \RR,dz+\theta)$ given some technical assumptions on $(P,d\theta)$. In particular, the construction works when $(P,d\theta)$ is symplectomorphic to the completion of a Liouville domain.

\subsection{Legendrian contact homology}
\label{sec:lch} We will use $\mathcal{Q}(\Lambda)$ to denote the set of double-points of the image of $\Lambda$ under the canonical projection $P \times \RR \to P$, also called \emph{Reeb chords} on $\Lambda$. We will now give a sketch of the definition of Legendrian contact homology in this setting, and we refer to \cite{ContHomP} for more details.

The Chekanov algebra $(\mathcal{A}(\Lambda;R),\partial)$ is a DGA associated to a Legendrian submanifold $\Lambda$. The underlying algebra is the unital, non-commutative, graded algebra freely generated by $\mathcal{Q}(\Lambda)$ over a ring $R$. Here we assume that the projection of $\Lambda$ is a generic immersion having transverse double-points, which means that $\mathcal{Q}(\Lambda)$ is a finite set. The grading of each generator is induced by the so-called Conley-Zehnder index.

Fix a generic compatible almost complex structure $J$ on $P$. The differential $\partial$ is defined on generators by a count of rigid $J$-holomorphic polygons in $P$ having boundary on the projection of $\Lambda$ and boundary-punctures mapping to double-points, of which exactly one is positive. The differential is then extended to all of $\mathcal{A}(\Lambda)$ by the Leibniz rule
\[\partial(ab)=\partial(a)b+(-1)^{|a|}a\partial(b).\]

We may always use coefficients $R=\ZZ_2$. In the case when $\Lambda$ is spin, we may also use $R=\ZZ$ after fixing a spin structure (see \cite{OrientLeg}). Hence we can work over an arbitrary ring $R$ in the latter case. It is also possible to define a richer invariant using so-called Novikov coefficients, where coefficients are chosen in the group ring $R[H_1(L)]$.

We have that $\partial^2=0$ by \cite[Lemma 2.5]{ContHomP}. By \cite[Theorem 1.1]{ContHomP} it follows that the stable-tame isomorphism type of the Chekanov algebra is invariant under the choice of a generic $J$ and Legendrian isotopy of $\Lambda$. In particular, its homology $HC_\bullet(\Lambda;R)$ is a Legendrian isotopy invariant.

\subsection{A linear complex over the characteristic algebra}
\label{sec:char}
Given a decomposition $\Lambda=\Lambda_1 \cup \Lambda_2$, where $\Lambda_i$ is not necessarily connected, we will use $\mathcal{Q}(\Lambda_1,\Lambda_2)$ to denote the Reeb chord starting on $\Lambda_1$ and ending on $\Lambda_2$. By the topology of the discs in the definition of the differential, for $c \in\mathcal{Q}(\Lambda_1,\Lambda_2)$ we have that $\partial(c)$ is a linear combination of words each containing at least one generator from $\mathcal{Q}(\Lambda_1,\Lambda_2)$ as well. 

In particular, the natural free $\mathcal{A}(\Lambda_2;R)\otimes\mathcal{A}(\Lambda_1;R)^{op}$-submodule (note the order!)
\[ \mathcal{A}(\Lambda_1,\Lambda_2;R):= \oplus_{c \in \mathcal{Q}(\Lambda_1,\Lambda_2)} \mathcal{A}(\Lambda_1;R) c \mathcal{A}(\Lambda_2;R) \subset \mathcal{A}(\Lambda_1 \cup \Lambda_2;R)\]
generated by $\mathcal{Q}(\Lambda_1,\Lambda_2)$ is a sub-complex, i.e.~we have used the natural identification
\[ \mathcal{A}(\Lambda_1;R) c \mathcal{A}(\Lambda_2;R) \simeq \mathcal{A}(\Lambda_2;R) \otimes \mathcal{A}(\Lambda_1;R)\]
that changes the order of the factors. The above invariance result implies that the homotopy type of $\mathcal{A}(\Lambda_1,\Lambda_2;R)$ is invariant under Legendrian isotopy.

Obviously, $\partial$ is restricted to $\mathcal{A}(\Lambda_1,\Lambda_2;R)$ is not $\mathcal{A}(\Lambda_2;R)\otimes\mathcal{A}(\Lambda_1;R)^{op}$-linear in general. We will amend this problem by taking a suitable quotient of the coefficient ring.

In the following we will assume that $\mathcal{A}(\Lambda_1;R) \simeq \mathcal{A}(\Lambda_2;R)$ are isomorphic as DGAs. This is for example true if $\Lambda_2$ is a small enough perturbation of a translation of $\Lambda_1$ in the $z$-coordinate. We use
\[\mathcal{C}_{\Lambda_i;R}:=\mathcal{A}(\Lambda_i;R)/\mathcal{A}(\Lambda_i;R)(\mathrm{im} \partial)\mathcal{A}(\Lambda_i;R),\:\:i=1,2\]
to denote the so-called \emph{characteristic algebra} of $\mathcal{A}(\Lambda_i;R)$, which is a quotient by the two-sided ideal generated by the boundaries. This invariant was introduced and studied in \cite{Computable} by Ng. Observe that a DGA is non-trivial if and only if its characteristic algebra is non-trivial, since we are considering unital algebras.

Since $\mathcal{C}_{\Lambda_1;R} \simeq \mathcal{C}_{\Lambda_2;R}$ by assumption, we will use $\mathcal{C}_R$ to denote either of them. We will consider the natural (non-free) $\mathcal{C}_R \otimes_R \mathcal{C}_R^{op}$-module
\[\mathcal{C}(\Lambda_1,\Lambda_2;R):=\oplus_{c \in \mathcal{Q}(\Lambda_1,\Lambda_2)}\mathcal{C}_Rc.\]
Consider the surjective algebra map
\[\Phi \colon \mathcal{A}(\Lambda_2;R) \otimes_R \mathcal{A}(\Lambda_1;R)^{op} \to \mathcal{C}_R \otimes_R \mathcal{C}_R^{op},\]
which for generators $c_i \in \mathcal{Q}(\Lambda_i)$, $i=1,2$, takes the form
\begin{gather*}
\Phi(c_2 \otimes 1) = [c_2] \otimes [1],\\
\Phi(1 \otimes c_1) = [1] \otimes [-c_1].
\end{gather*}

Under the algebra map $\Phi$, the boundary $\partial$ restricted to $\mathcal{A}(\Lambda_1,\Lambda_2;R)$ descends to a differential $\partial_{\mathcal{C}}$ on the $\mathcal{C}_R \otimes_R \mathcal{C}_R^{op}$-module $\mathcal{C}(\Lambda_1,\Lambda_2;R)$. Finally, the Leibniz rule implies that $\partial_{\mathcal{C}}$ is $\mathcal{C}_R \otimes_R \mathcal{C}_R^{op}$-linear.

\begin{thm}[Theorem 1.1 \cite{ContHomP}]
\label{thm:invariance}
The homotopy type of $(\mathcal{C}(\Lambda_1,\Lambda_2;R),\partial_\mathcal{C})$ is independent of the choice of a regular almost complex structure and invariant under Legendrian isotopies of $\Lambda_1 \cup \Lambda_2$ that induce no births or deaths of Reeb chords on either $\Lambda_i$.
\end{thm}
\begin{proof}
This result is just an algebraic consequence of the invariance result \cite[Theorem 1.1]{ContHomP} which, in turn, depends on \cite[Section 4.3]{OrientLeg}.

The isomorphism class of the characteristic algebra $\mathcal{C}_{\Lambda_i;R}$ is obviously independent of the isomorphism class of the Chekanov algebra $(\mathcal{A}(\Lambda_i;R),\partial)$ (see Theorem \cite[Theorem 3.4]{Computable} for a more refined result). The latter isomorphism class is independent of the choice of a generic compatible almost complex structure and invariant under Legendrian isotopy of $\Lambda_i$, given that there are no births or deaths of Reeb chords on $\Lambda_i$, as follows by \cite[Lemma 4.13]{OrientLeg}.

We have shown that the coefficients in the complex $(\mathcal{C}(\Lambda_1,\Lambda_2;R),\partial_\mathcal{C})$ are invariant under the deformations in the statement of the theorem. That the homotopy type of the complex itself is invariant is now a straight-forward algebraic consequence of the invariance result for the Chekanov DGA.
\end{proof}

\subsection{Construction of the exact Lagrangian immersion $L_\Lambda$}
\label{sec:double}
In the following we assume that the Legendrian submanifold $\Lambda \subset J^1(M)$ admits an exact Lagrangian cap $L_{\Lambda,\emptyset}$ inside the symplectisation
\[(\RR \times J^1(M), d(e^t(dz+\theta_M))\]
which moreover is cylindrical in the set $\{ t \le 1 \}$. Observe that there is a symplectic identification
\begin{gather*}
(\RR \times T^*M \times \RR, d(e^t(dz+\theta_M))) \to (T^*(\RR_{>0} \times M),d\theta_{\RR_{>0} \times M}),\\
(t,(\mathbf{q},\mathbf{p}),z) \mapsto ((e^t,\mathbf{q}),(z,e^t\mathbf{p})).
\end{gather*}

Let $\pi_{Lag}(\Lambda) \subset T^*M$ denote the image of $\Lambda$ under the corresponding canonical projection and observe that this is an exact Lagrangian immersion. Consider the exact Lagrangian immersion 
\[ [-1,1] \times \Lambda \to [-1,1] \times \RR \times T^*M \simeq T^*([-1,1] \times M) \]
satisfying the following.
\begin{itemize}
\item The image under the canonical projection to $[-1,1] \times T^*M$ coincides with $[-1,1] \times \pi_{Lag}(\Lambda)$.
\item The image intersected with $\{0\} \times \RR \times T^*M$  coincides with $\{0\} \times \{ 0 \} \times \pi_{Lag}(\Lambda)$.
\item The pull-back of the Liouville form to $[-1,1] \times \Lambda$ has a primitive on the form $(1+\varphi(t))g$, where $g$ is a real-valued function on $\Lambda$ and $\varphi \colon [-1,1] \to \RR_{>0}$ is an even non-zero Morse function which coincides with $\pm t$ in a neighbourhood of $t=\pm 1$ and which has its only critical point at $t=0$ (which thus is a non-degenerate minimum).
\end{itemize}
See \cite[Section 10.2]{ContHomR} and \cite[Section 4.3.2]{OrientLeg} for the construction of a similar exact Lagrangian immersion. Let $L_\Lambda$ denote the closed manifold obtained by gluing the cap to itself along its end. It is possible to extend the above exact Lagrangian immersion to an exact Lagrangian immersion
\[ \varphi \colon L_\Lambda \to \RR \times \RR \times T^*M \simeq T^*(\RR \times M)\]
whose image is invariant under the symplectomorphism
\[ ((t,z),(\mathbf{p},\mathbf{q})) \mapsto ((-t,-z),(\mathbf{p},\mathbf{q}))\]
and which furthermore satisfies
\begin{itemize}
\item $\varphi(L_{\Lambda}) \cap ([-1,1] \times \RR \times T^*M)$ coincides with the above Lagrangian immersion.
\item $\varphi$ has no double-points outside of $\{ 0\} \times \RR \times T^*M$.
\item $\varphi(L_{\Lambda}) \cap T^*([1,+\infty) \times M)=L_{\Lambda,\emptyset} \cap T^*([1,+\infty) \times M)$.
\end{itemize}
For the last part, we have used the above identification of the symplectisation with the corresponding cotangent bundle.

\begin{lma}{\cite[Lemmas 4.14, 4.15]{OrientLeg}} Every $i \oplus J$-holomorphic disc in
\[ \CC \times T^*M \simeq \RR \times T^*M \times \RR\]
having boundary on $L_\Lambda$ and exactly one positive puncture is contained inside the plane $\{ t=0 \}$.
\end{lma}
\begin{proof}
This follows from the argument in \cite[Lemmas 4.14, 4.15]{OrientLeg} together with monotonicity argument in the proof of Proposition \ref{prop:width} (more precisely, from the part which shows the non-uniruledness of a closed Lagrangian submanifold constructed using an exact Lagrangian cap).
\end{proof}

\begin{lma}
\label{lem:equiv}
The Legendrian contact homology DGAs of $L_\Lambda$ and $\Lambda$ with $\ZZ_2$-coefficients (and no Novikov coefficients) are homotopy equivalent. If $L_{\Lambda,\emptyset}$ is a spin cobordism and if the DGA of $\Lambda$ is induced by a spin structure on the cap, then the same is true with $\ZZ$-coefficients.
\end{lma}
\begin{proof}
First, observe that there is a canonical identifications of the generators which preserves the grading. We must thus show that the differentials agree.

By the above Lemma we know that every $i \oplus J$-holomorphic disc with boundary on $L_\Lambda$ lies in the plane $\{ t =0\}$, and thus, is contained inside $\{(0,0)\} \times T^*M$ and has boundary on $\{(0,0)\} \times \pi_{Lag}(\Lambda)$. This shows that there is a bijective correspondence between the $i \oplus J$-holomorphic discs contributing to the differential of $L_\Lambda$ and the $J$-holomorphic discs contributing to the differential of $\Lambda$. Comparing the signs, it follows that the Legendrian contact homology DGA of the Legendrian lift of $L_\Lambda$ is equal to the Legendrian contact homology DGA of $\Lambda$ for these choices of almost complex structures. 
\end{proof}

\subsection{Proof of Theorem \ref{thm:uniruled}}
This proof is an adaptation of the proof of \cite[Theorem 5.5]{DualityLeg} to the current algebraic set-up.

Take the two-copy lift as constructed in \cite[Section 3.1]{DualityLeg}. In other words, in a Weinstein neighborhood of $L_1:=L \subset (P,d\theta)$ symplectomorphic to the co-disc bundle $D^*L$ of some radius, we let $L_2$ be given by the section $df$ for a sufficiently $C^1$-small Morse function
\[ f \colon L \to \RR.\]
We choose a Legendrian lift of $L_1 \cup L_2$ to the contactisation $P \times \RR$ where the $z$-coordinate of the lifts satisfy \[\min_{z}(L_2)-\max_{z}(L_1)>N>0.\] 
In particular, we may assume that
\[\mathcal{Q}(L_1 \cup L_2)=\mathcal{Q}(L_1) \sqcup \mathcal{Q}(L_2) \sqcup \mathcal{Q}(L_1,L_2)\]
and
\[ \mathcal{Q}(L_1,L_2)=\mathcal{P} \sqcup Crit(f) \sqcup \mathcal{Q} \]
where $\mathcal{P}$ and $\mathcal{Q}$ each can be naturally identified with $\mathcal{Q}(L)$ and where, for $N>0$ sufficiently large, the action of the Reeb chords in $\mathcal{P}$ are strictly less then the action of the Reeb chords in $Crit(f)$ which, in turn, are strictly less than the action of the Reeb chords in $\mathcal{Q}$.

We will consider the complex
\[ (\mathcal{C}(L_1,L_2;R),\partial_{\mathcal{C}}),\]
defined in Section \ref{sec:char} above, which is a $\mathcal{C}_R \otimes \mathcal{C}_R^{op}$-module over the characteristic algebra of the Chekanov DGA of $L_1 \cup L_2$. Since $L$ is displaceable by assumption, this complex is acylic by Theorem \ref{thm:invariance}.

Given that $f$ is sufficiently $C^1$-small and $N>0$ sufficiently large, since $\partial_{\mathcal{C}}$ preserves the action filtration, we get a decomposition
\begin{gather*}
\mathcal{C}(L_1,L_2;R)=Q \oplus C \oplus P, \\
\partial_{\mathcal{C}}=\begin{pmatrix} \partial_Q & 0 & 0 \\
\rho & -\partial_C & 0 \\
\eta & \sigma & \partial_P
\end{pmatrix},
\end{gather*}
where $C$ is generated by the Reeb chords $Crit(f)$ corresponding to the critical points of $f$, and where $Q$ and $P$ are generated by $\mathcal{Q}$ and $\mathcal{P}$, respectively.

By \cite[Proposition 3.7(2)]{DualityLeg}, for a suitable choice of compatible almost complex structure $J$, we moreover have an isomorphism
\[(C_\bullet,\partial_C) \simeq (C^{Morse}_{\bullet+1}(f)\otimes \mathcal{C}_{L;R},\partial_f)\]
where the latter is the Morse complex (with degree shifted), for some generic choice of metric on $L$, having coefficients in the characteristic algebra $\mathcal{C}_{L;R}$.

Let $c_L \in C$ denote an $\mathcal{C}_R$-fundamental class of $L$, which is a linear combination of the the maxima of $f$. Observe that in the case when $1+1\neq 0$ in $\mathcal{C}_R$, and hence in $R$, the closed manifold $L$ is spin by assumption. In particular $L$ is orientable in this case, and there is such a fundamental class.

Assume that $L$ is not uniruled. In particular, we may assume that there is a generic compatible almost complex structure $J$ on $P$ for which there are no $J$-holomorphic disks with boundary on $L$, one positive boundary-puncture at $c$, and one boundary point passing through a maximum $c_M$ of $f$. Since \cite[Theorem 3.6]{DualityLeg} implies that the coefficient in front of $c_M$ in the expression $\rho(c)$ is determined by a count of such rigid configurations, also called generalised pseudo-holomorphic discs, it follows that $c_L$ cannot be in the image of $\rho$.

Since $c_L$ is not in the image of $\partial_C$ either, as follows by degree reasons, we have concluded that $c_L$ is not a boundary.

We now argue that $\partial_{\mathcal{C}}(c_L)=0$. First, $\partial_C(c_L)=0$ holds since $c_L$ is the image of an $R$-fundamental class under the above morphism.

It remains to show that $\sigma(c_L)=0$. This follows by a count of discs analogous to the count in the proof of \cite[Theorem 5.5]{DualityLeg} where the "dual" statement is shown. For a Reeb chord $c \in \mathcal{Q}(L)$, let $p_c$ denote the corresponding Reeb chord in $\mathcal{P}$ and $c_i \in \mathcal{Q}(L_i)$ the corresponding Reeb chord on $L_i$ for $i=1,2$ (so $c_1=c$).

Let $c_M \in C$ be a local maximum of $f$. \cite[Theorem 3.6(3)]{DualityLeg} implies that the $J$-holomorphic discs contributing to $\sigma(c_M)$ correspond to generalised pseudo-holomorphic discs with boundary on $L$ consisting of
\begin{itemize}
\item A $J$-holomorphic disc in $P$ of expected dimension $-1$ with boundary on $L$ and one positive puncture.
\item A negative gradient flow-line connecting $c_M$ to the boundary of the above disc.
\end{itemize}
Formula (3.11) in \cite{DualityLeg} implies that the above $J$-holomorphic disc is constant. In particular, for generic $J$, it is contained entirely in a double-point and must have exactly one positive and one negative puncture at this Reeb chord as follows by index and action considerations.

For each $c \in \mathcal{Q}(L)$, and for a generic $f$, it follows that there are two $J$-holomorphic discs contributing to $\partial_{C}(c_M)$: one with negative punctures at $p_c$ and $c_1$, and one with negative punctures at $p_c$ and $c_2$. In other words
\[ \partial_{\mathcal{C}}(c_M)=\sum_{c \in \mathcal{Q}(L)}(\Phi(c_2 \otimes 1)+\Phi(1 \otimes c_1))p_c=\sum_{c \in \mathcal{Q}(L)}(c-c)p_c=0.\] 
Also see \cite[Remark 5.6]{DualityLeg} for the similar cancellation in the proof of \cite[Theorem 5.5]{DualityLeg}.

By the acyclicity of $(\mathcal{C}(L_1,L_2;R),\partial_{\mathcal{C}})$, the coefficients must be trivial, i.e $\mathcal{C}_{L;R}=0$, since otherwise $[c_M] \in H\mathcal{C}(L_1,L_2;R)$ would represent a non-zero class in homology. It follows that
\[ HC(L;R)=0 \]
as well, since $1 \in \mathcal{A}(L;R)$ is a boundary.

\section*{Acknowledgements}
The author is grateful to Matthew Strom Borman, Octav Cornea, Mark McLean, Emmy Murphy, and Kai Zehmisch for interesting discussions and suggestions, and to Baptiste Chantraine for valuable discussions about the fundamental class in Legendrian contact homology. The author would finally like to thank the organisers of the conference \emph{D-Days: A Panorama of Geometry} in ETH, Z\"{u}rich, a conference celebrating Dietmar Salmon's 60:th birthday, during which the main parts of this article were written.

\bibliographystyle{alphanum}
\bibliography{references}
\end{document}